\documentclass{amsart}

\usepackage{graphicx}

\usepackage{mathpazo}
\usepackage{microtype}
\newcommand{\PP}{\mathbb{P}}
\newcommand{\Aut}{\textrm{Aut}}
\newcommand{\C}{\mathbb{C}}
\newcommand{\R}{\mathbb{R}}
\newcommand{\Z}{\mathbb{Z}}

\newtheorem{theorem}{Theorem}[section]

\newtheorem{proposition}[theorem]{Proposition}

\theoremstyle{definition}
\newtheorem{definition}[theorem]{Definition}
\newtheorem{example}[theorem]{Example}

\theoremstyle{remark}

\numberwithin{equation}{section}

\begin{document}

\title{Hurwitz numbers, ribbon graphs, and tropicalization}
\author{Paul Johnson}
\address{
Mathematics Department
Columbia University
Room 509, MC 4406
2990 Broadway
New York, NY 10027}
\curraddr{}
\email{paul.da.johnson@gmail.com}
\thanks{The author was supported in part by NSF Postdoctoral Fellowship DMS-0902754}


\begin{abstract}
The double Hurwitz number $H_g(\mu,\nu)$ has at least four equivalent definitions.  Most naturally, it counts the covers of the Riemann sphere by genus $g$ curves with certain specified ramification data.  This is classically equivalent to counting certain collections of permutations.  More recently, it has been shown to be equivalent to a count of certain ribbon graphs, or as a weighted count of certain labeled graphs.

This note is an expository account of the equivalences between these definitions, with a few novelties.  In particular, we give a simple combinatorial algorithm to pass directly between the permutation and ribbon graph definitions.
The two graph theoretic points of view have been used to give proofs that $H_g(\mu,\nu)$ is piecewise polynomial in the $\mu_i$ and $\nu_j$. We use our algorithm to compare these two proofs.
\end{abstract}

\maketitle

\section{Introduction}

Hurwitz theory is the study of ramified covers of curves;  Hurwitz numbers count the number of covers having specified ramification.  This paper aims to clarify the connections between four equivalent definitions of the \emph{double Hurwitz number} $H_g(\mu,\nu)$.  We name each definition with the initial letter of what it counts:
\begin{enumerate}
\item[(C)] As a count of certain ramified \emph{covers}
\item[(P)] As a count of certain sets of \emph{permutations}
\item[(RG)] As a count of certain labeled \emph{ribbon graphs}
\item[(TG)] As a weighted count of certain labeled \emph{(tropical) graphs}
\end{enumerate}
Definition (C) in terms of covers and Definition (P) in terms of permutations are classical, as is their equivalence through the monodromy of the cover.   Definitions (RG) and (TG) in terms of graphs are more modern.

Definition (RG) was first given by Goulden, Jackson and Vakil in \cite{GJV}, adapting ideas already used for single Hurwitz numbers \cite{Arnold, OPH} to double Hurwitz numbers.

Definition (TG) was introduced in \cite{CJM1}.  This definition is inspired by tropical geometry, though it can be understood without it.  We use only the cartoon summary of tropicalization: it degenerates Riemann surfaces into graphs (also known as tropical curves).  Figure \ref{fig-tropicalization} shows a double Hurwitz cover and its tropicalization.

\begin{figure}[!h]
\caption{A cartoon of tropicalization} \label{fig-tropicalization}
\includegraphics[width=5in]{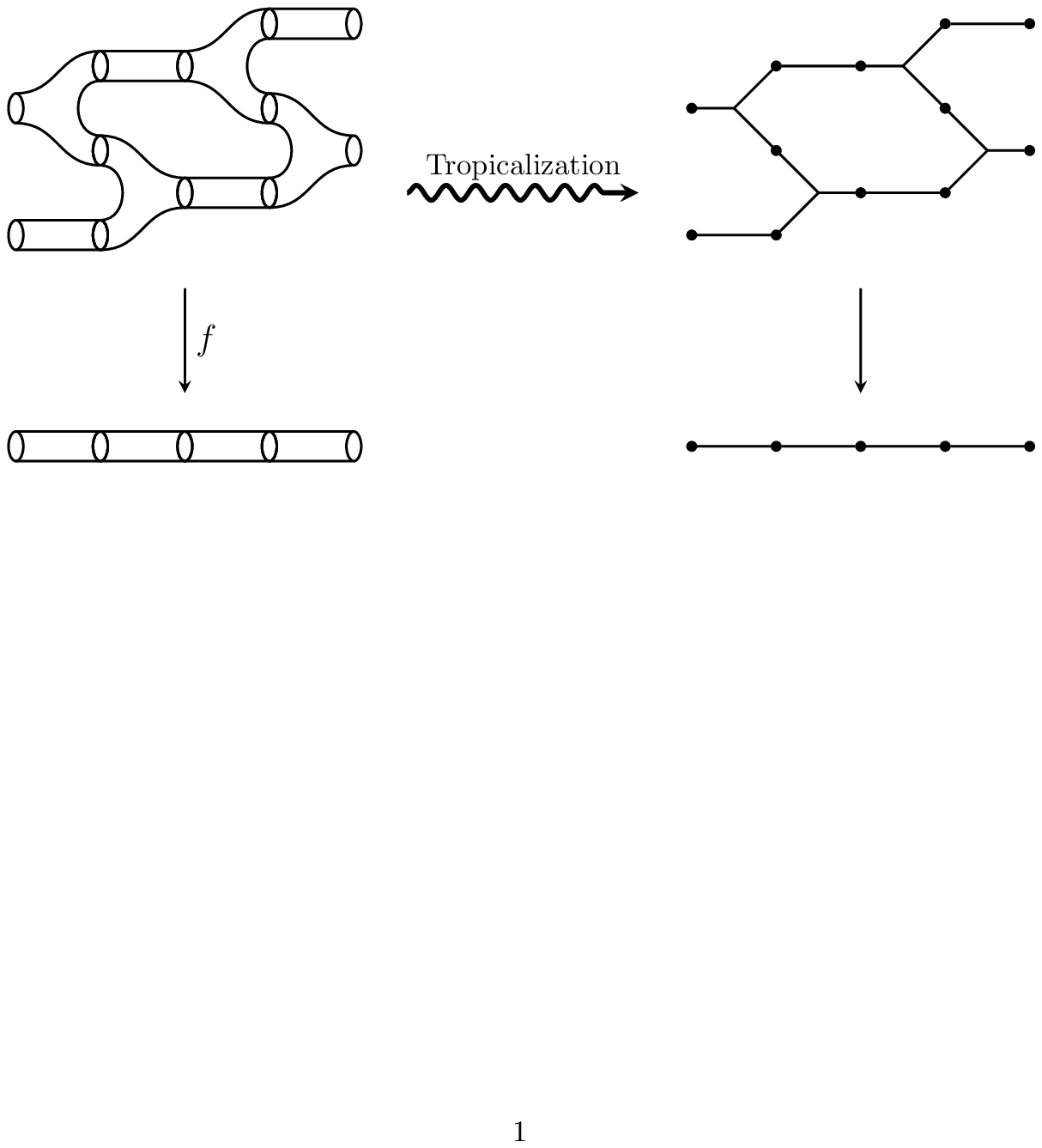}
\end{figure}

The modern definitions were introduced to help reveal structure in the double Hurwitz numbers $H_g(\mu,\nu)$.  In \cite{GJV}, Definition (RG) is used to prove that $H_g(\mu,\nu)$ is a piecewise polynomial function; Definition (TG) is used in \cite{CJM1, CJM2} to give another proof of this fact.  Both proofs use Ehrhart theory in a similar way.   There is now a third proof of piecewise polynomiality using Definition (P) in \cite{wedge}, that first uses the classical step of using representation theory, and then follows Okounkov \cite{OHur} in encoding the resulting combinatorics in terms of operators acting on the infinite wedge.  This approach is powerful, but loses all contact with geometry, and is beyond the scope of this article.

The main goal of this paper is to survey these definitions and their interconnectedness, and thus it is largely expository.  There are several novelties in the exposition: for instance, we find it conceptually useful to use Morse theory, and we use slightly different ribbon graphs than those in \cite{GJV}.  The biggest original contribution, however, is to illustrate \emph{direct} equivalences between some of the definitions.

Though all four definitions are known to be equivalent, the actual equivalences can be rather circuitous.  In \cite{GJV} Definition (RG) is shown to be equivalent to Definition (C), while in \cite{CJM1} Definition (TG) is shown to be equivalent to Definition (P).  This makes it difficult to compare the graph theoretic definitions, and hence difficult to compare the two proofs of piecewise polynomiality.  Our main new contribution is a direct combinatorial algorithm to pass from Definition (RG) to Definition (P), which leads to a direct equivalence between Definition (RG) and (TG).  We use this last equivalence to compare the two proofs of piecewise polynomiality, which was the initial motivation for this paper.

In the remainder of the introduction, we give some motivation for studying double Hurwitz numbers.  In particular, we explain why this paper is included in a volume about integrable systems.  Section \ref{sec-classical} recalls the classical Definitions (C) and (P) of Hurwitz numbers and their relationship.  Section \ref{sec-ribbon} recalls the ribbon graph definition (RG), and shows it is equivalent to the geometric definition (C), while Section \ref{sec-ribbon-permutations} contains the algorithm connecting ribbon graphs and permutations.  Finally, Section \ref{sec-tropical} introduces the tropical definition (TG), discusses its relation to the other definitions, and compares the two proofs of piecewise polynomiality.

\subsection{Motivation}

Classically, Hurwitz theory was used to show qualitative results about $\mathcal{M}_g$, the moduli space of curves.  Riemann used it in his calculation of its dimension (see \cite{GH}, page 255) and Hurwitz and others used it to show it was irreducible (see \cite{fulton}).

More recently, Hurwitz theory has been used to give quantitative information about $\overline{\mathcal{M}}_{g,n}$, the compactification of the moduli space of pointed curves.  The ELSV formula \cite{ELSV} expresses certain intersection numbers in $H^*(\overline{\mathcal{M}}_{g,n})$ to single Hurwitz numbers $H_g(\mu)$ to great effect.  Coupled with Okounkov's result \cite{OHur} connecting Hurwitz numbers with integrable hierarchies of KP type, the ELSV formula was used in \cite{OPH} to prove the Witten-Kontsevich conjecture that certain intersections on $\overline{\mathcal{M}}_{g,n}$ are governed by the KdV hierarchy.  The proof in \cite{OPH} has been streamlined and extended in \cite{KL, kazarian}.

Double Hurwitz numbers have also been used to study $H^*(\overline{\mathcal{M}}_{g,n})$, see for instance \cite{GJVF}.  However, a recent motivation for studying double Hurwitz numbers has been the conjecture of Goulden, Jackson, and Vakil \cite{GJV} that there should be an ELSV-type formula where single Hurwitz numbers are replaced with double Hurwitz numbers, and the $\overline{\mathcal{M}}_{g,n}$ is replaced with some $\overline{\textrm{Pic}}_{g,n}$ parameterizing curves together with a complex line bundle.

We do not deal further with these issues here, except to note that a major piece of motivation for the conjecture of \cite{GJV} is the piecewise polynomiality of $H_g(\mu,\nu)$, and that the graph theoretic proofs of this fact appear as if they could be shadows of the geometric structure conjectured in \cite{GJV}.

\section{The classical viewpoints: covers and permutations} \label{sec-classical}

 In this section we review the classical perspectives on the double Hurwitz number $H_g(\mu,\nu)$ in terms of ramified covers and the symmetric group, and briefly indicate how they are equivalent.   We also introduce Morse theory to the study of double Hurwitz numbers, which will prove useful later.

For the rest of this paper, $\mu$ and $\nu$ are partitions of $d$.  The partitions $\mu$ and $\nu$ have lengths $\ell(\mu)=m$ and $\ell(\nu)=n$ -- that is, $\mu_1+\cdots+\mu_m=d$, where the $\mu_i$ are positive integers.  Let $r=2g-2+m+n$.

\subsection{Definition of double Hurwitz numbers in terms of covers}

\begin{definition}[Covers]
A $(\mu,\nu,g)$-\emph{Hurwitz cover} is a degree $d$ map $f:\Sigma\to \PP^1$ from a genus $g$ connected complex curve $\Sigma$ to $\PP^1$, satisfying
\begin{enumerate}
\item  $f$ has ramification profile $\mu$ over $0$ and $\nu$ over $\infty$
\item  $f$ has simple ramification over $r$ additional fixed points $p_i \in \PP^1$
\item  $f$ has no other ramification
\item The $m$ elements of $f^{-1}(0)$ and the $n$ elements of $f^{-1}(\infty)$ are labeled
\end{enumerate}

If $f:\Sigma\to\PP^1$ and $g:\Sigma^\prime\to\PP^1$ are Hurwitz covers, then an isomorphism $h:f\to g$ is an isomorphism $h:\Sigma\to\Sigma^\prime$ that satisfies $hg=f$ and preserves the labels of the marked points over $0$ and $\infty$.
 \end{definition}

\begin{definition}
The double Hurwitz number $H_g(\mu,\nu)$ is the count of $(\mu,\nu,g)$-Hurwitz covers, where each cover $f$ is counted with weight $1/|(\Aut(f)|$.
\end{definition}

A few brief comments are in order.  The Riemann-Hurwitz formula gives the formula $r=2g-2+m+n$; it is a special feature of double Hurwitz numbers that $r$ is independent of $d$.   Labeling the preimages of $0$ and $\infty$ is a convention used in \cite{GJV}.  Using this convention only changes the result by a factor of $|\Aut(\mu)|\cdot |\Aut(\nu)|$.

\subsection{Definition of double Hurwitz numbers in terms of permutations}
We now give the definition in terms of permutations.

First, we define a \emph{labeled permutation}.  Suppose the cycle decomposition of $\sigma$ has $k$ cycles.  Then a labeling of $\sigma$ is a bijection between the cycles and the set $\{1,\dots,k\}$.  Thus, we can talk about the $i$th cycle of a labeled permutation.

\begin{definition}[Permutations]
A \emph{$(\mu,\nu,g)$-monodromy set} is an element $$(\sigma_0,\tau_1,\dots,\tau_r,\sigma_\infty)\in S_d^{r+2},$$
together with a labeling of $\sigma_0$ and $\sigma_\infty$, satisfying:
\begin{enumerate}
\item $\sigma_0$ and $\sigma_\infty$ have cycle types $\mu$ and $\nu$ respectively
\item The $\tau_i$ are all transpositions
\item $\sigma_0\cdot \left(\prod_{i=1}^r \tau_i\right)\cdot \sigma_\infty=1$
\item The group generated by the $\tau_i$ and $\sigma_j$ acts transitively on $\{1,\dots, d\}$
\end{enumerate} 
\end{definition}

\begin{proposition} \label{prop-classical}
$H_g(\mu,\nu)$ is $1/d!$ times the number of $(\mu,\nu,g)$-monodromy sets.
\end{proposition}

Proposition \ref{prop-classical} is classical; we give some indication of the proof in Section \ref{sec-covertopermutations}.  

Dividing by $d!$ essentially comes from the fact that the $d$ sheets of the cover are not labeled, but to produce monodromy data we need a labeling of the set.  Relabeling the sheets corresponds to having $S_d$ act by simultaneous conjugation of all $r+2$ coordinates on $S_d^{r+2}$, where we label the first and last permutations as above.  Then $H_g(\mu,\nu)$ counts the number of $S_d$ orbits of labeled $S_d^{r+2}$-tuples, where each orbit $\mathcal{O}$ is counted with weight $\frac{1}{|G_\mathcal{O}|}$, where $G_\mathcal{O}$ is the stabilizer group of an orbit.  This viewpoint can create a slightly stronger version of Proposition \ref{prop-classical}.  First, define:

\begin{definition}
If $(\sigma^\prime_0,\tau^\prime_1,\dots,\tau^\prime_r,\sigma^\prime_\infty)$ is another $(\mu,\nu,g)$-monodromy set, an isomorphism between them is an element $g\in S_d$ such that $g^{-1}\sigma_i g=\sigma_i^\prime$ for $i=0$ or $\infty$, $g^{-1}\tau_j g=\tau_j^\prime$ for $j\in \{1,\dots, r\}$, and such that this conjugation action preserves the labels on $\sigma_0$ and $\sigma_\infty$.
\end{definition}

Then the proof of Proposition \ref{prop-classical} actually creates a bijection between isomorphism classes of Hurwitz covers and isomorphism classes of monodromy data.  Furthermore, for each isomorphism class of objects, it creates an isomorphism of the corresponding automorphism groups.  Put another way.  it creates an equivalence of categories between the groupoid of Hurwitz covers and the groupoid of monodromy data.

In much of what follows, rather than working with the $S_n$ orbit of the tuple $(\sigma_0,\tau_1,\tau_2,\dots,\tau_r)$, we find it simpler to define permutations $\sigma_i$ for $1\leq i \leq r$ by $\sigma_i=\tau_i\tau_{i-1}\cdots\tau_1\sigma_0$.  Clearly the tuple $(\sigma_0, \sigma_1,\dots, \sigma_r)$ determines the tuple $(\sigma_0,\tau_1,\dots,\tau_r, \sigma_\infty)$ and vice-versa in a way that commutes with conjugation by $S_d$.

\subsection{Equivalence between covers and permutations} \label{sec-covertopermutations}

The equivalence of these two definitions is classical and can be found in many places.   We do not present this in full, but we give a short and incomplete review of how to go from a geometric cover to a set of permutations, largely because we take a slightly unusual perspective that we will find useful later.   Rather than produce the transpositions $\tau_i$, we produce the set of permutations $(\sigma_0,\dots,\sigma_r)$ defined above.

Let $\Sigma^\circ=f^{-1}(\C^*)$; alternatively, $\Sigma^\circ$ is $\Sigma$ minus the $m+n$ marked points that are the preimages of $0$ and $\infty$. Let $f_\circ$ denote $f$ restricted to $\Sigma^\circ$. Fix the $r$ points of simple ramification so that they map to the points $1,2,\dots,r\in\C^*$.  Consider the set of negative real numbers $\R_-\subset\C^{*}$.  Since $\R_-$ misses all the critical values of $f_\circ$ and $f_\circ$ has degree $d$, we see that $f_\circ^{-1}$  has $d$ components, each isomorphic to the real line.  Arbitrarily choose a labeling of them with the numbers $1,\dots,d$.

Consider now the $r+1$ circles $Y_k\subset\C^*, k\in\{0,\dots,r\}$, given by $$Y_k=\left\{z\in\C^*\big ||z|=k+1/2\right\}.$$  Give $Y_k$ the orientation as the boundary of the disk containing zero. Let $X_k=f^{-1}(Y_k)$.  Since the $Y_k$ miss the critical values of $f$, we have that each $X_k$ is the union of some number of circles.  Orient $X_k$ by lifting the orientation of $Y_k$.   Each component of $X_k$ corresponds to a cycle of the permutation $\sigma_k$, as we now describe.

Observe that $X_k$ intersects each of the $d$ lines in $f^{-1}_\circ(\R_-)$ transversely; label the $d$ points of intersection in $X_k\cap f^{-1}_\circ(\R_-)$ according to which component of $f^{-1}_\circ(\R_-)$ it belongs to.  Then the elements on a component of $X_k$ form a cycle of $\sigma_k$, with the cyclic ordering given by the orientation of $X_k$.  In other words, to find how $\sigma_k$ acts on $ i$, find the label $i$ on $X_k$, and then follow along $X_k$ with its natural orientation until we find the next labeled point of orientation -- say it's labeled $j$.   Then we $\sigma_k\cdot i=j$.

Much of this discussion is clearly visible in Figure \ref{fig-tropicalization}.  The left hand side shows the map $f_\circ$.  The circles visible on $\C^*$ are the circles $Y_k$, and the circles visible on $\Sigma^\circ$ are the components of the $X_k$.

\subsection{Cut-Join and Morse theory} \label{sec-cutjoinmorse}

The above description would work, essentially unchanged, for completely arbitrary ramification.  One thing that is special about having simple ramification is that $g=|f_\circ|:\Sigma^\circ\to \R_+$, is a Morse function, and the critical values have Morse index 1.  Thus, when we pass a critical point, the manifolds $g^{-1}_\circ(0,x)$ change by attaching a 1-cell. 

We now connect the Morse-theory viewpoint to Cut-Join analysis, which we will find useful in the discussion of the tropical definition of $H_g(\mu,\nu)$ in Section \ref{sec-tropdef}.  We begin by recalling Cut-Join in the group theoretic context, and then use Morse-theory to explain the geometric meaning.

In terms of permutations, Cut-Join analysis studies how the cycle type of a permutation $\sigma$ changes when we multiply a transposition $\tau$.  Suppose that $\tau=(ij)$.  If $i$ and $j$ belong to the same cycle of $\sigma$, then that cycle is cut into two different cycles in $\tau\sigma$.  For example, take $(ij)=(13)$, and $\sigma=(123456)$; then we have $(13)(123456)=(12)(3456)$, and the one cycle of $\sigma$ has been cut in two.   If, however, $i$ and $j$ belong to two different cycles of $\sigma$, then those two cycles are joined into one cycle in $\sigma\tau$. To illustrate, take $\tau=(13)$ again, but now let $\sigma=(12)(3456)$.  We have $\tau\sigma=(123456)$, and the two cycles of $\sigma$ have been joined together.

Furthermore, if we know the lengths of the cycles that are being cut or joined, we can count the number of possibilities for $\tau$.  There are always $k\ell$ transpositions that join a $k$ cycle and an $\ell$ cycle into a $k+\ell$ cycle, there are $k+\ell$ different transpositions that split a $k+\ell$ cycle into a $k$ cycle and an $\ell$ cycle when $k\neq \ell$, and there are $k$ transpositions that split a $2k$ cycle into two cycles of length $k$.  

This can be seen as follows.   Suppose our transposition is $(ij)$, and that it joins a $k$-cycle and an $\ell$-cycle.  Then one of $i,j$, suppose $i$, must be one of the $k$ elements in the $k$-cycle, but which one is unconstrained.  Similarly, $j$ must be one of the $\ell$ elements in the other cycle, and so we have $k$ choices for $i$ and $\ell$ choices for $j$. 

In the other case, to split a $k+\ell$ cycle into a $k$-cycle and an $\ell$ cycle we can choose any of the $k+\ell$ elements in the cycle to be $i$.   Then $j$ must be whatever appears $k$ steps after $i$ in the cycle.  When $k=\ell$, each transposition counted this way appears twice.  

We now use the Morse function viewpoint to explain Cut-Join in the context of geometry.  Recall that the geometric analog of multiplying by a transposition is attaching a $1$-cell.  If the two boundary points of the 1-cell lie on the same component of $X_k$, that component is cut into two components in $X_{k+1}$; if the boundary points of the 1-cell lie on different components of $X_k$, those components are joined into one component in $X_{k+1}$.  These two cases of cut and join correspond to whether the ``waist'' of the pair of pants we are adding faces left or right.  

The multiplicity with which these possibilities happen can also be seen geometrically.  If a component of $X_k$ maps to $Y_k$ with degree $d$, the procedure in Section \ref{sec-covertopermutations} constructs puts $d$ labels on that component, which divide the circle into $d$ intervals.  When we want to attach a one cell to this boundary component, we thus have $d$ different choices of places to attach it.

\section{Double Hurwitz numbers and labeled ribbon graphs} \label{sec-ribbon}

In this section we introduce certain labeled ribbon graphs called $(\mu,\nu,g)$-\emph{Hurwitz ribbon graphs}.  We call them $(\mu,\nu,g)$-HRGs for short, or just HRGs when $\mu,\nu$ and $g$ are not specified.    In Section \ref{sec-cover-ribbon} we show that counting $(\mu,\nu,g)$-HRGs gives the double Hurwitz number $H_g(\mu,\nu)$, and in Section \ref{sec-pp} we use this to show that $H_g(\mu,\nu)$ is piecewise polynomial.  Finally, in Section \ref{sec-GJVribbon} we explain the relationship between our ribbon graphs and those used in \cite{GJV}.

\subsection{Ribbon graphs}

Intuitively, a ribbon graph is a graph whose edges have been thickened to be ribbons.  There are many equivalent formal definitions of ribbon graphs.  We use the following:

\begin{definition} A \emph{ribbon graph} is a pair $\Gamma\subset \Sigma$ of a graph $\Gamma$ embedded in an oriented topological surface $\Sigma$, so that each component of $\Sigma\setminus \Gamma$ is a disk.
\end{definition}

Ribbon graphs are a very natural concept and is studied under many different names.  They are sometimes referred to as fat graphs, or simply graphs on surfaces.  Though we require the surfaces to be oriented, in general, unorientable ribbon graphs can be studied.

A ribbon graph naturally defines a cell complex on the surface $\Sigma$.  The vertices and edges of the cell complex are the vertices and edges of the graph $\Gamma$. The faces are the components of $\Sigma\setminus\Gamma$.  In what follows, we talk about the vertices, edges, and faces of a ribbon graph, and write $v\prec e, e\prec f, v\prec f$ to say that a given vertex, edge or face $v,e$ or $f$ is incident to another.

We are interested in \emph{bicolored ribbon graphs}, also known as hypermaps.  In a bicolored ribbon graph, the faces are colored gray and white so that adjacent faces do not have the same color - thus, each edge separates a gray face from a white face.   This provides an orientation on the edges, namely, travel along the edge so that the gray face is on the right and the white face is on the left.  We refer to this as the natural orientation of the edges.

\subsection{Ribbon graphs from Hurwitz covers}

In this section, given a Hurwitz cover $f:\Sigma\to \PP^1$, we construct a bicolored ribbon graph $\Gamma\subset\Sigma$ with labeled vertices and faces.  The definition of $(m,n,g)$-ribbon graphs encapsulates the resulting structure.  An $(m,n,g)$-ribbon graph does not capture all the information contained in the cover $f$ -- to do that we introduce edge weight in Section \ref{sec-weighting}.  However, $(m,n,g)$-ribbon graphs  play a role in the proof of piecewise polynomiality in Section \ref{sec-pp}.

\begin{figure}[h]
\caption{Various markings on $\PP^1$} \label{fig-markedsphere}
\includegraphics[width=3in]{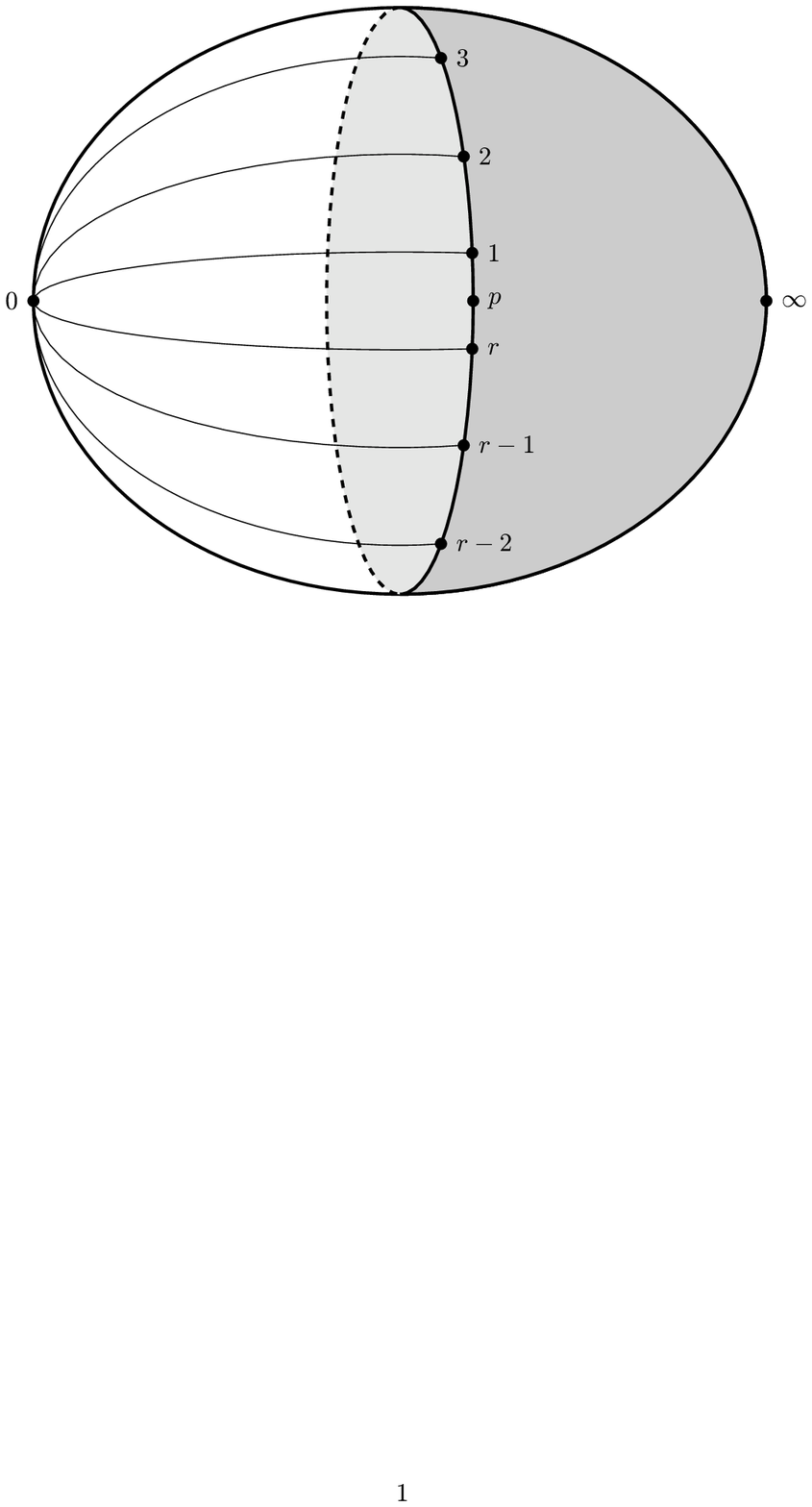}
\end{figure}

 First, suppose that $f:\Sigma\to\PP^1$ is a $(\mu, \nu, g)$ Hurwitz cover, with simple ramification at the $r$ roots of unity.  We label the $r$ roots of unity $1,\dots,r$ following the standard counterclockwise ordering as the boundary of the unit disk around 0, as in Figure \ref{fig-markedsphere}.

Let $U\subset\PP^1$ be the unit circle $|z|=1$.  The graph $\Gamma\subset\Sigma$ is the inverse image of the unit circle: $\Gamma=f^{-1}(U)$. Away from the $r$ points of ramification on $\Sigma$, $f$ is a local homeomorphism, and so away from the $r$ points of ramification $f^{-1}(U)$ is locally homeomorphic to an interval.  These intervals are the edges of $\Gamma$.
  
Now let $x\in\Sigma^\circ$ be a point of ramification.  Since around $x$ the map $f$ is equivalent to $z\mapsto z^2$, we see that locally near $x, f^{-1}(U)$ is homeomorphic to the union of the coordinate axes in $\mathbb{R}^2$.  Thus, every ramification point $x$ is a 4-valent vertex of $\Gamma$.  We label the vertices of $\Gamma$ with $\{1,\dots,r\}$ according to which root of unity it maps to.

To verify that $\Gamma\subset\Sigma$ is a bicolored ribbon graph, we must check that each component of $\Sigma\setminus\Gamma$ is a disk, and produce a coloring of the faces.  To produce the coloring, we lift a coloring on $\PP^1\setminus U$.  Let $H_0$ be the component of $\PP^1$ containing 0, and color it white, and let $H_\infty$ be the component containing $\infty$, and color it gray, as in Figure \ref{fig-markedsphere}.  Now, each component of $\Sigma\setminus\Gamma$ must map to either $H_0$ or $H_\infty$, and we color it white or gray according to whether it maps to $H_0$ or $H_\infty$.  It is immediate that this is a bicoloring.

To check that each face is a disk, observe that the only ramification of $f$ on $\Sigma\setminus\Gamma$ occurs over $0$ and $\infty$.  A map to the disc only ramified over $0$ must be a disjoint union of maps of the form $z\mapsto z^{k_i}$, for some $k_i$. Thus, each component of $\Sigma\setminus\Gamma$ must be a disk.  

We see that each face contains exactly one ramification point, and the ramification points are labeled.  Thus there must be $m$ white faces, labeled $1,\dots, m$ according to which ramification point it contains.  Similarly there are $n$ labeled gray faces.  The following definition encapsulates the structure we have defined on $\Sigma$:

\begin{definition}
A $(m,n, r)$-\emph{ribbon graph} is a 4-valent, bicolored ribbon graph with $r$ labeled vertices, $m$ labeled white faces and $n$ labeled gray faces.
\end{definition}

\subsection{Edge weights on ribbon graphs} \label{sec-weighting}

Note that the definition of a $(m,n,r)$-ribbon graph does not take into account the partitions $\mu$ and $\nu$.  Given a $(\mu,\nu,g)$ Hurwitz cover $f:\Sigma\to\PP^1$ we introduce an additional edge weighting on $\Gamma$ that allows us to reconstruct $f$.

Given a Hurwitz cover $f:\Sigma\to\PP^1$, construct a $(m,n,r)$-ribbon graph $\Gamma\subset\Sigma$ as in the previous section.  Then, on  the unit circle $U\subset \PP^1$, mark an additional point $p$ between the point marked $r$ and the point marked $1$, as in Figure \ref{fig-markedsphere}.  For each edge $e$ of $\Gamma$ define $w(e)$ be the number of points in $f^{-1}(p)$ that lie on $e$.

We now discuss two properties that any edge weighting $w(e)$ defined as above has automatically.  
  
First, as $f$ has degree $d$, we see that the sum of all the $w(e)$ must be $d$.  We can refine this as follows: since $f$ restricted to the $i$th white face has degree $\mu_i$, we have that the sum of the $w(e)$ for the edges $e$ around the $i$th white face must be $\mu_i$.  Similarly the sum of the $w(e)$ for the edges around the $j$th gray face $g_j$ must be $\nu_j$.

\begin{definition}
An edge weighting $w(e)$ on an $(m,n,g)$-ribbon graph is $(\mu,\nu)$-\emph{balanced} if for each white face $f_i$ and each gray face $g_j$ we have:
\begin{align*}
\sum_{e\prec f_i}w(e)&=\mu_i&  \sum_{e\prec g_j}w(e)&=\nu_j.
\end{align*}
\end{definition}

Second, it is clear from the definition of $w(e)$ that it is always a nonnegative integer.  However, if we look closer we can see that the weights of certain edges must be strictly positive.  Suppose that with the standard orientation, the edge $e$ goes from vertex $i$ to vertex $j$.  If $i<j$, then it is possible that $e$ has weight zero.  However, if $i\geq j$, then $e$ must have at least one preimage of $p$, and so $w(e)>0$.

\begin{definition} \label{def-positive}
An edge weighting $w(e)$ is \emph{positive} if $w(e)$ is always a nonnegative integer, and $w(e)>0$ when $e$ is an edge from $i$ to $j$ with $i\geq j$.
\end{definition}

Naming this condition ``positive'' is best explained in terms of the following definition:

\begin{definition} \label{def-length}
Given an $(m,n,r)$-ribbon graph with edge weighting $w(e)$, define the length $\ell(e)$ of an edge $e$  by:
$$\ell(e)=2\pi\left(w(e)+\frac{j-i}{r}\right).$$
\end{definition}

Definition \ref{def-length} of $\ell(e)$ is natural in the context of the ribbon graphs and weights we have constructed so far from Hurwitz covers.  Give the sphere the standard round metric, so that the unit circle has length $2\pi$.  Then, each edge of the ribbon graph $\Gamma\subset\Sigma$ inherits a length by defining $f$ to be an isometry away from the points of ramification.  This is exactly the length $\ell(e)$ we have just defined.  Definition \ref{def-positive} of positivity for $w(e)$ is equivalent under Definition \ref{def-length} to simply asking $\ell(e)>0$.

\subsection{Computing $H_g(\mu,\nu)$ with ribbon graphs} \label{sec-cover-ribbon}

We package the ribbon graph structure we have produced from a Hurwitz cover $f:\Sigma\to\PP^1$ into the following definition:

\begin{definition}[Ribbon Graphs]
A $(\mu,\nu,g)$-\emph{Hurwitz ribbon graph}, or $(\mu,\nu,g)$-HRG,  is an $(m,n,r)$-ribbon graph together with an edge weighting $w(e)$ that is positive and $(\mu,\nu)$-balanced.

An isomorphism between HRG is an isomorphism of the underlying ribbon graphs (i.e., the underlying cell complexes) that preserves the labels of the vertices and faces and the weights $w(e)$ of the edges.
\end{definition}

We now show that the definition of $H_g(\mu,\nu)$ in terms of ribbon graphs is equivalent to that in terms of covers.  This is the essentially the argument contained in \cite{GJV, OPH}.
 
\begin{proposition} \label{prop-rg}
$H_g(\mu, \nu)$ is the count of isomorphism classes of weighted ribbon graphs, where each weighted ribbon graph $\Gamma$ is counted with weight $1/|\Aut(\Gamma)|$.
\end{proposition}

\begin{proof}
We have seen how to construct a $(\mu,\nu,g)$-Hurwitz Ribbon graph from a $(\mu,\nu,g)$-Hurwitz cover.  Furthermore, from the construction it is clear that an automorphism of the cover give rise to an automorphism of the resulting ribbon graph.

We now indicate how the construction can be reversed.  Suppose $\Gamma\subset\Sigma$ is a HRG; from this data we construct a Hurwitz cover $f:\Sigma\to\PP^1$.   

We first describe $f$ restricted to $\Gamma$, which maps to the unit circle $U\subset\PP^1$.  We do this by putting the standard round metric on $\PP^1$, and giving each edge of $\Gamma$ the length $\ell(e)$ described in the previous section.  Then we define $f$ by mapping vertex $i$ to the $i$th root of unity, and mapping each edge $e$ to the unit circle $U$ in the unique way that is isometric and preserves orientation.  The fact that $w(e)$ is $(\mu,\nu)$-balanced and positive guarantees this is well defined.

We now extend the map $f$ to the faces.  The balanced condition also guarantees that the length of the boundary of the $i$th white face is $2\pi\mu_i$.  Since $f$ is an isometry, we see that the boundary of the $i$th white face  maps to $U$ as a $\mu_i$-fold cover.  By the Riemann existence theorem there is a unique holomorphic extension that maps to $|z|\leq 1$ with degree $\mu_i$, ramification $\mu_i$ over $0$, and no other ramification.  The extension of $f$ to the gray faces is analogous.  It is clear that automorphisms of the HRG produce automorphisms of the cover, and that this construction is inverse to the construction of an HRG from a Hurwitz cover.
\end{proof}

As was the case with Proposition \ref{prop-classical}, the proof actually proves a slightly stronger statement, in that it constructs an equivalence of groupoids between Hurwitz covers and Hurwitz Ribbon graphs.

\subsection{Piecewise Polynomiality} \label{sec-pp}

In \cite{GJV}, Goulden, Jackson and Vakil used Proposition \ref{prop-rg} and Ehrhart theory to show that $H_g(\mu,\nu)$ is a piecewise polynomial function of the variables $\mu_i$ and $\nu_j$.  We briefly recall this proof now.

The main idea is to group the $(\mu,\nu,g)$-HRG's together by forgetting the edge weightings $w(e)$ and considering the underlying $(m, n,r)$-ribbon graph.  Working in reverse, given a $(m,n,r)$-ribbon graph $\Gamma\subset \Sigma$, consider the space of possible edge weightings $w(e)$ that would make it into a $(\mu,\nu,g)$-HRG.  This is a subset of the lattice given by the $\mathbb{Z}$-span of the space of edges.  Requiring that $w(e)$ is $(\mu,\nu)$-balanced imposes $m+n$ linear equations on the lattice points.  Requiring that $w(e)$ is positive imposes a linear equality $w(e)\geq 0$ or $w(e)>0$ for each edge $e$.  Thus, the space of possible edge weights form the lattice points in a (partially open, due to positivity) polytope.

We see that changing the values of $\mu$ and $\nu$ results in parallel translating the hyperplanes of this polytope.  Ehrhart theory implies that as we parallel translate the faces of this lattice polytope, the number of lattice points in it varies as a piecewise polynomial function.  Since $H_g(\mu,\nu)$ can be calculated as the sum of the number of lattice points in a finite number of lattice polytopes, we see that $H_g(\mu,\nu)$ must be piecewise polynomial.

With slightly more work, one can use this method to determine that these polynomials have degree $4g-3+m+n$, and that the walls of polynomiality are given by equalities of the form $$\sum_{i\in I} \mu_i=\sum_{j\in J}\nu_j$$ for some subsets $I, J$ of $[m], [n]$, respectively.

\subsection{Comparison with GJV's ribbon graphs} \label{sec-GJVribbon}

The ribbon graph description we have given above differs slightly from the one used by Goulden, Jackson and Vakil in \cite{GJV}, which we call GJV ribbon graphs. We now briefly indicate the relationship between GJV ribbon graphs and HRGs.

A GJV ribbon graph has $m$ labeled vertices, $n$ labeled faces, and $r=2g-2+m+n$ labeled edges.  As with our HRGs, they can be constructed from a Hurwitz cover by lifting a structure on the sphere, as we now describe.

Let the $r$ ramification points happen over the $r$ roots of unity.  Draw a line on $\PP^1$ joining $0$ to each of the $r$ roots of unity, as shown in Figure \ref{fig-markedsphere}.  On the cover $\Sigma$, over the $i$th root of unity there is a distinct point $p_i$ where the map $f$ has simple ramification.  Because $f$ has simple ramification at $p_i$, the line connecting the $i$th root of unity with 0 has exactly two lifts to $\Sigma$ passing through $p_i$.   Both preimages of this line terminate  at one of the $m$ preimages of zero.  The union of these two is thus an edge between two of our vertices.   One can check that the defining $\Gamma$ to be the union of these edges indeed gives a ribbon graph structure to $\Sigma$.  

Here is an alternate description of $\Gamma$.  Begin by defining $\Gamma^\prime$ to be the inverse image of the star shaped graph on $\PP^1$ used above, with edges connecting $0$ to each of the $r$ roots of unity.  We now construct $\Gamma$ by simplifying $\Gamma^\prime$.  The inverse image of each root of unity consists of $d-1$ points on $\Gamma^\prime$.  Of these, $d-2$ are univalent vertices; delete these vertices and their incident edges.  The other preimage is a two valent vertex.  Delete this vertex, and merge the two incident edges into a single edge.  The resulting ribbon graph is $\Gamma$.

The weights on GJV's ribbon graphs are associated to ``corners'' instead of edges.   A corner of a ribbon graph is as it sounds: formally, it is a point of incidence between a vertex and a face.  Note that a vertex might be incident to the same face multiple times, giving multiple corners.   The corner weightings must satisfy similar balancing and positivity conditions.

There is a natural operation on ribbon graphs known as taking the medial graph that when applied to a GJV ribbon graph gives the corresponding HRG.
To construct the medial graph, a vertex is placed at the midpoint of each edge.  Across every corner of the ribbon graph, we draw an edge connecting these vertices.  These are the edges and vertices of the medial graph.  The medial graph is thus always four valent and bicolored, with the two colors of faces corresponding to vertices and faces of the old graph.   

 Figure \ref{fig-medialgraph} shows a planar graph and its medial graph.

\begin{figure}[!h]
\caption{The medial graph of a ribbon graph}  \label{fig-medialgraph}
\includegraphics[width=4in]{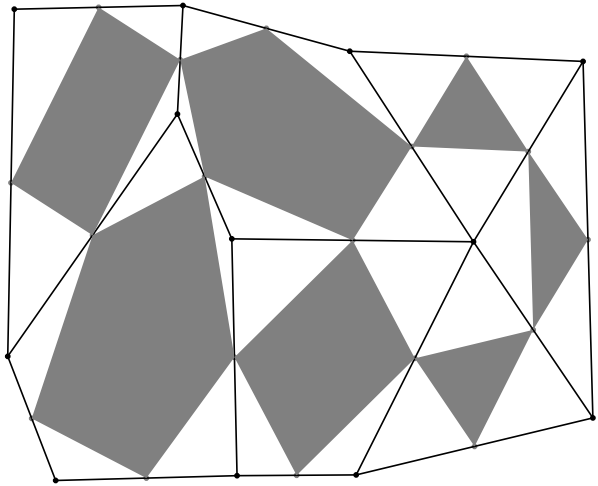}
\end{figure}

\section{From ribbon graphs to permutations} \label{sec-ribbon-permutations}
We have seen how to pass between permutations and covers, and between covers and ribbon graphs; it is clear that composition gives a way to pass between permutations and ribbon graphs.  In this section we describe a simple algorithm to pass directly from a ribbon graph to a set of permutations, bypassing the construction of the cover.

Our algorithm is entirely combinatorial, but the geometry of the cover is still be visible.   In particular, in understanding the algorithm it is useful to recall our construction of permutations from a geometric cover, and our understanding of $|f|$ as a Morse function on $\Sigma^\circ$, as in Section \ref{sec-covertopermutations} and Figure \ref{fig-tropicalization}. 
 We present the algorithm Section \ref{sec-algorithm}, and then describe the geometric meaning of the algorithm in Section \ref{sec-morsealgorithm}, which makes clear that our algorithm is the composition of the equivalences already presented.

\subsection{The algorithm} \label{sec-algorithm}

As in Section \ref{sec-covertopermutations}, rather than produce the transpositions $\tau_i$, we produce the permutations $\sigma_i=\tau_i\tau_{i-1}\cdots\tau_1\sigma_0$.

The first step of the algorithm is to place $w(e)$ tick marks on each edge $e$, and then to chose an arbitrary bijection between the resulting $\sum w(e)=d$ tick marks and some index set $I$ of size $d$.  Our permutations $\sigma$ act on $I$.  The arbitrary labelling is essentially encoding the fact that the permutations $\sigma_i$ are determined only up to simultaneous conjugation: a different choice of labeling corresponds to acting by conjugation on the tuple of permutations.

Calculating $\sigma_i$ from a ribbon graph is similar to finding $\sigma_i$ as in Section \ref{sec-covertopermutations}.  To find $\sigma_i\cdot x$, the basic idea is to find the tick mark labeled $x$, then trace along the edge following its natural orientation until we reach the next tick mark   -- $x$ maps to whatever this tick mark is labeled.  For example, in Figure \ref{fig-circles}, any $\sigma_i$  maps $a$ to $b$.

\begin{figure}[!h]
\caption{From ribbon graphs to permutations}  \label{fig-circles}
\includegraphics[width=5in]{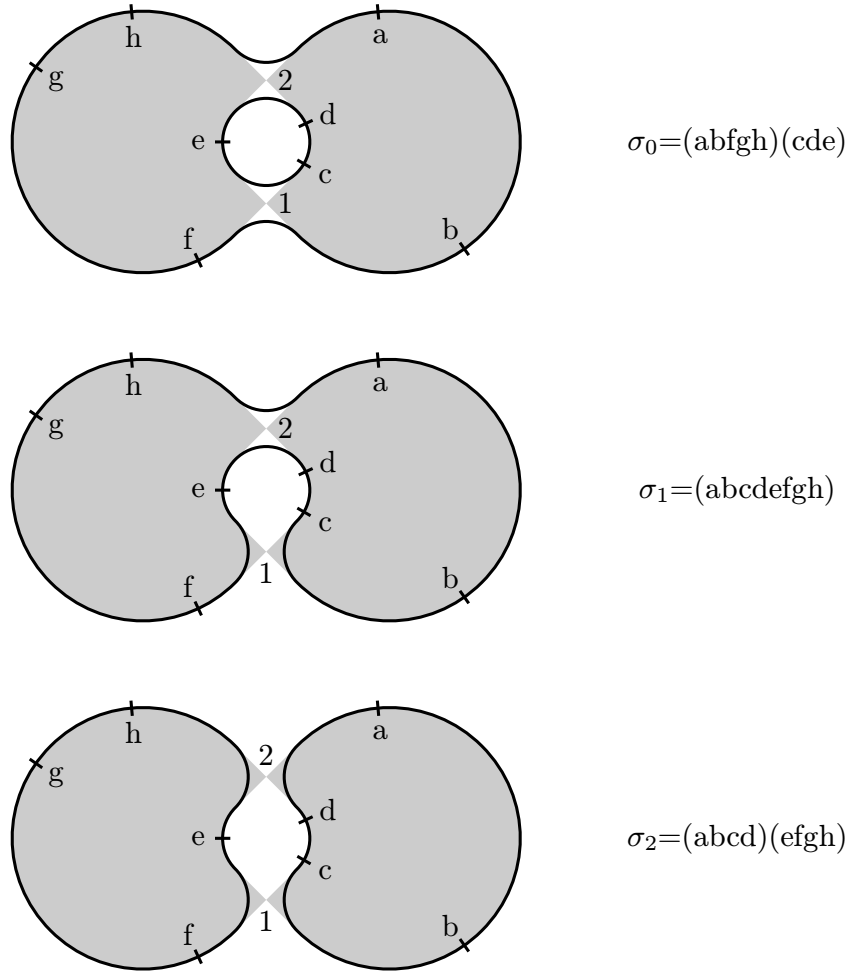}
\end{figure}

The difficulty is that this only makes sense if there is a tick mark between $x$ and the next vertex of $\Gamma$.  Otherwise, we need a rule to determine whether we turn left or right at that vertex.  In Figure \ref{fig-circles}, it is unclear what $b$ should map to -- following the edge we reach vertex 2 before we reach any tick marks.

We solve this by giving a ``traffic rule'' at each vertex, telling us to either turn left or right when we reach that vertex, and then continue along the new edge until we reach the first tick mark.  To produce the permutation $\sigma_i$, the traffic rules are the following: turn left if the vertex label is greater than $i$, and turn right if the vertex label is less than or equal to $i$.

So, for example, the first permutation $\sigma_0$ is given by turning left at every intersection, or equivalently, following the boundary of the gray faces.   The last permutation $\sigma_r$ is given by following the boundary of white cells -- turning right at every intersection.  To pass from $\sigma_{i}$ to $\sigma_{i+1}$, the only thing that changes is that the traffic rule at vertex $i+1$ changes from ``turn left'' to ``turn right''.

To further illustrate, in Figure \ref{fig-circles}, we have $\sigma_0\cdot b=f$, but $\sigma_1\cdot b=\sigma_2\cdot b=c$.

\begin{example}
An example of the algorithm is essentially contained in Figure \ref{fig-circles} - we give a brief description here.  On the left hand side of Figure \ref{fig-circles} are three copies of the same $(\mu,\nu,0)$-HRG, where $\mu=4+4$ and $\nu=5+3$.  In each copy, the partition of the edges into oriented cycles is shown in  a thicker black line.  On the right hand side, the resulting tuple of permutations is shown.
\end{example}

We briefly indicate how to reverse the algorithm, and construct a $(\mu,\nu,g)$-HRG from the series of permutations $\sigma_0,\dots,\sigma_r$.  Begin by taking a white disk for each cycle in $\sigma_0$ and labeling it with the appropriate label.  Then a add $\mu_i$ tick marks to the boundary of the $i$th disk, labeled with the elements in the $i$th cycle of $\sigma_0$ appearing in counter-clockwise order.  Finally, put a black line around the boundary of each of the white circles.

Now, to pass from $\sigma_0$ to $\sigma_1$, we know that one cycle is cut or two are joined.  This  corresponds to adding a vertex to our ribbon graph.  In case one cycle gets cut, the vertex lies on the same white circle in two different places; in case two cycles are joined it lies on two separate white circles.  After placing the vertex, change the thick black line as occurs in our algorithm.  There is a unique place to add the vertex that  changes $\sigma_0$ to $\sigma_1$.

  The algorithm for adding the vertices continues analogously for all $\sigma_r$.  When all vertices have been added, and the thick black line moved the final time, we glue a gray cell to each component of the thick black line, creating a closed surface $\Sigma$.

\subsection{Geometry of the algorithm} \label{sec-morsealgorithm}

Having a direct relationship between the ribbon graph and symmetric group points of view, we would now like to also tie in the geometric angle. The Morse theory perspective introduced in Section \ref{sec-covertopermutations} is useful here.

To begin, we note that the connection between the geometric and ribbon graph perspectives we have described is in conflict with the geometry of the Morse function picture.   In our construction of the ribbon graphs, we chose the cover $f$ so that the points $p_i$ with simple ramification all map to the unit circle, and so $|f(p_i)|=1$.  Thus, for this choice of $|f|$ is not a Morse function -- we need all the critical points to have distinct critical values.   We fix this by using the labelling of the $p_i$ to slightly deform the map $f$ we used to construct the ribbon graphs.  We keep the phase of each $f(p_i)$ the same but change the norms slightly so that $|f(p_1)|<|f(p_2)|<\cdots < |f(p_r)|$.

Recall from \ref{sec-covertopermutations} how the permutation $\sigma_i$ were visible in Figure \ref{fig-tropicalization}: the $i$th column of circles on $\Sigma$ are the cycles of $\sigma_i$.

Now assume that we know the permutation $\sigma_{i-1}$, and want to know $\sigma_i$. In terms of permutations, this corresponds to multiplying by $\tau_i$.  In the geometric picture, the cycles of the $\sigma_i$ correspond to components of the level sets $|f|^{-1}(i)$, and to understand multiplication by a transposition, we have to understand how the level sets of a Morse function change as we pass through a critical value.  We have seen that geometrically this corresponds to adding a pair of pants, and that we have a cut-join analysis corresponding to which way the pair of pants is oriented.

We would now like to connect the geometry of the Morse function point of view to the changing traffic rules of our algorithm.  Consider the local picture of the level sets of an index-1 Morse singularity, as shown in Figure \ref{fig-morselevels}.   Here, we have a saddle point drawn, with three level sets -- at, below and above the critical level -- drawn on the surface and projected beneath the surface.  The level set at the critical value gives a four valent vertex, as we have seen in our construction of the ribbon graphs.  Note that the projected level sets exactly model how the local picture around a vertex when we change a traffic rule, as in Figure \ref{fig-circles}.

\begin{figure}[h] 
\caption{Level sets of a Morse singularity} \label{fig-morselevels}
\includegraphics[width=4in]{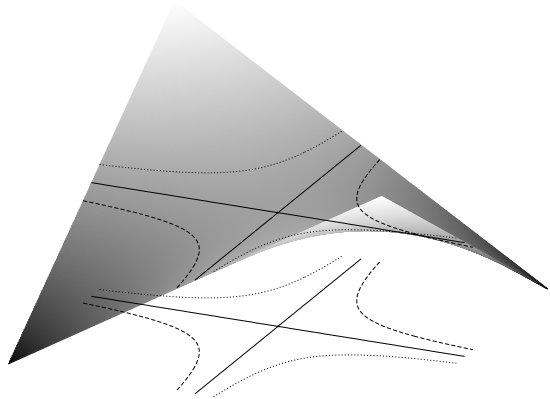}
\end{figure}

\section{The tropical definition} \label{sec-tropical}

In this section, we discuss the tropical definition of double Hurwitz numbers and its relationship to the other definitions. 

Section \ref{sec-tropdef} introduces the definition of a tropical graph, and Section \ref{sec-mondef} refines them to monodromy graphs, which are tropical graphs with certain edge weights.  The relationship between tropical graphs and monodromy graphs is analogous to the relationship between ribbon graphs and Hurwitz ribbon graphs.

The double Hurwitz number $H_g(\mu,\nu)$ is as a weighted count of monodromy graphs.  The weight a given monodromy graph is counted with includes an automorphism factor, as with the other three methods.  However, it also includes a separate factor of the product of all the interior edge weights, called the tropical multiplicity.  In Section \ref{sec-tropmult}, we explain the meaning of this tropical multiplicity in terms of the two classical viewpoints.

Finally, in Section \ref{sec-polytopemap} we relate the ribbon graph definition (RG) to the tropical definition (TG).  In particular, the tropical definition also leads to a proof of piecewise polynomiality via lattice points in polytopes.   We show that tropicalization maps the ribbon graph polytopes to the tropical polytopes by linear projections of a nice form.

\subsection{Tropical graphs} \label{sec-tropdef}

Intuitively, a $(m,n, r)$-tropical graph carries the information captured in the cartoon view of tropicalization in Figure \ref{fig-tropicalization}, where cylinders become edges, and pairs of pants become trivalent vertices. More precisely:

\begin{definition}

An $(m,n,r)$-\emph{tropical graph}, is a directed graph with $m$ univalent sources, labeled 1 to $m$, $n$ univalent sinks labeled 1 to $n$, and $r$ trivalent vertices, labeled $1$ to $r$.  The ordering of the trivalent vertices is compatible with the directions of the edges: if there is an edge from vertex $i$ to vertex $j$, then $i<j$.
\end{definition}

A tropical graph is shown on the right hand side of Figure \ref{fig-tropicalization} -- the ordering of the vertices is induced by their horizontal positioning.  From Figure \ref{fig-tropicalization}, we can also see the geometric meaning of the tropical graph: it encodes the combinatorics of the Morse function $|f_\circ|$ described in Section \ref{sec-cutjoinmorse}.  The edges represent components of the level sets of $|f_\circ|$, and the graph encodes the combinatorics of how these level sets are glued together by the 1-cells.   Our previous discussion also immediately explains the meaning of the graph in terms of permutations: the tropical graph encodes the combinatorics of how the cycles of the permutations $\sigma_i$ are split apart and joined together.

A simple calculation shows that a $(m,n,r)$-tropical graph has genus $g$, by which we mean its first homology group is $g$-dimensional.  For example, in Figure \ref{fig-tropicalization} we see a genus 1 curve tropicalizes to a graph with one loop.

\subsection{Monodromy graphs} \label{sec-mondef}

We now introduce monodromy graphs, which are tropical graphs with edge weights $w(e)$ satisfying certain properties.  

We first motivate the edge weights by explaining what information they capture in terms of the classical definitions.  First we take the viewpoint of counting covers.  Recall from our cartoon view of tropicalization that each edge $e$ of a tropical graph represents a cylinder in $\Sigma$ mapping to $\C^*$.  Each such map naturally has a degree (In the notation of Section \ref{sec-covertopermutations} the degree with which a given component of $X_k$ maps to $Y_k$), and the edge weight $w(e)$ is just this degree.   

In terms of permutations, recall that each edge represents a cycle in some permutation $\sigma_k$.  The edge weight $w(e)$ encodes the lengths of these cycles. 

From the classical viewpoints,  we see that the edge weights $w(e)$ must satisfy some obvious constraints, which we call the balancing conditions.  In terms of the geometry of the cover, the balancing conditions simply say that the degree must be preserved.    We define a $(\mu,\nu,g)$-monodromy graph to be an $(m,n,r)$-tropical graph with an edge weighting satisfying the balancing conditions.

\begin{definition}
A \emph{$(\mu,\nu,g)$-monodromy graph} is an $(m,n,r)$-tropical graph where edge $e$ has a weight $w(e)$ satisfying the following balancing conditions: \begin{itemize}
\item The edge adjacent to the $i$th source has weight $\mu_i$ 
\item The edge adjacent to the $j$th sink has weight $\nu_j$
\item For each interior vertex $v$, the sum of weights of the incoming edges equals the sum of the weights of the outgoing edges 
\end{itemize}

An isomorphism of monodromy graphs is an isomorphism of graphs that preserves vertex labels and edge weights.
\end{definition}

The edge weights can be intuitively understood as the flow of water along the directed graph.  The balancing conditions say that at that at the source and sink vertices determined amounts of water enter or leave the graph, while at the interior vertices water is conserved.

\subsection{Tropical multiplicity} \label{sec-tropmult}

It was shown in \cite{CJM1} that the Hurwitz number $H_g(\mu,\nu)$ is a weighted sum over all monodromy graphs.  However, in contrast with the previous definitions, the weight each graph is counted with is not simply an automorphism factor; there is an additional factor known as the \emph{tropical multiplicity}, which counts how many complex objects map under tropicalization to a given tropical object.  We begin by discussing tropical multiplicity in terms of the two classical definitions of Hurwitz numbers. 

Geometrically, we have seen that the tropical graph encodes the combinatorics of the level sets of the Morse function $|f_\circ|$, and the edge weights encode the degree with which each component of the level set maps to $\C^*$.  The information captured in a monodromy graph does not allow us to recreate a cover, and the failure to do so is exactly the freedom captured by Cut-Join analysis.  When we attach the boundary of a one cell to a degree $d$-edge, there are $d$ different ways to do so.  In other words, in tropicalization, forgets any twisting of the cylinders, and the tropical multiplicity records how many different ways this twisting could have happened. 

In terms of permutations, observe that the monodromy graph does not capture the full information of the orbit of $(\sigma_0,\cdots,\sigma_r)$ under simultaneous conjugation.  It does capture the cycle type of each $\sigma_i$, or, equivalently, the orbit of $(\sigma_0,\cdots,\sigma_r)$ under independent conjugation on each factor  --  that is, the orbit under the action of $(S_d)^{r+1}$, where the $i$th factor acts by conjugation on $\sigma_i$.  Actually, a monodromy graph carries slightly more information than this: it also records which cycles get split and joined together by the transpositions.  Up to this extra bookkeeping, the tropical multiplicity counts how many diagonal $S_d$ orbits a given $(S_d)^{r+1}$-orbit splits into.

Let $E^\circ(\Gamma)$ denote the set of interior edges of a monodromy graph, i.e., those not adjacent to a univalent vertex.  Then the tropical multiplicity of a monodromy graph is the product of the edge weights of all interior edges.

\begin{proposition} \label{prop-tropical}
The double Hurwitz number $H_g(\mu,\nu)$ can be calculated as a weighted sum over all monodromy graphs, where each monodromy graph $\Gamma$ is counted with weight 
$$\frac{1}{|\Aut(\Gamma)|} \prod_{e\in E^\circ(\Gamma)} w(e).$$
\end{proposition}

Proposition \ref{prop-tropical} is easily verified by repeated use of the Cut-Join analysis that was reviewed in Section \ref{sec-cutjoinmorse}.  We give a brief outline from the geometric perspective.  A more detailed proof, from the permutation perspective, can be found in \cite{CJM1}.  

We need to determine how many Hurwitz covers tropicalize to a given monodromy graph.  The monodromy graph determines which components of $X_k$ the 1-cells of $|f_\circ|$ are attached to, and how many labels each circle carries.   The information that is lost is where on each circle the labels are attached, and this is exactly the information the Cut-Join analysis counts.  Each time we attach a one cell, we multiply by the degree of the edge that was cut, or the degrees of the two edges that were joined.  The only subtlety is that the labels on the original $m$ components are indistinguishable, and so we do not multiply by their weights.  This immediately yields the tropical multiplicity.

\subsection{Tropicalization as a map between polytopes} \label{sec-polytopemap}

The tropical point of view was used in \cite{CJM1,CJM2} to give another proof that $H_g(\mu,\nu)$ is piecewise polynomial.  This proof is quite similar to the one using ribbon graphs, in that they both use Ehrhart theory applied to polytopes associated to graphs.   In this section, we explain a connection between these two proofs.  We begin by briefly recall the tropical proof of piecewise polynomiality.

The basic idea is the same as that of the ribbon graph proof: for a given $(\mu,\nu,g)$ tropical graph $\Gamma$,the space of possible edge weights forms a polytope, called the \emph{flow polytope} in \cite{CJM2} based on the analogy with the flow of water.   We describe this polytope now.

Let $\mathbb{Z}[E^\circ(\Gamma)]$ denote the lattice of $\Z$-linear combinations of interior edges.  Then the space of possible edge weights $w(e)$ are given by two conditions.  First, we have the requirement that the coefficient of each edge is non-negative, which is a linear inequality.  Secondly, we have the balancing conditions, which are all linear equations in the coefficients of the edges.  Thus, the space of possible $(\mu,\nu,g)$-monodromy graphs with the same underlying $(m,n,r)$-tropical graph are the lattice points in an integral polytope, and changing the $\mu_i$ and $\nu_j$ results in parallel translation of the facets of this polytope.

In the tropical point of view, we are no longer just counting the lattice points, but counting each lattice point with its tropical multiplicity.  However, this multiplicity is a polynomial in the coordinates of the vector space, and so Ehrhart theory again tells us that this counting procedure produces a piecewise polynomial function in the $\mu_i$ and $\nu_i$.

We now use our algorithm for passing from ribbon graphs to permutations to relate the two proofs of piecewise polynomiality.   For each $(\mu,\nu,g)$-HRG tropicalization produces a monodromy graph.  Our first observation is that all HRGs with the same underlying $(m,n,r)$-ribbon graph $\Gamma$ map to monodromy graphs with the same underlying tropical graph $\Gamma^\prime$.  As the HRGs and monodromy graphs are both lattice points in a polytope, tropicalization then gives some map $\varphi$ between the ribbon graph polytope and the flow polytope.  Our second observation is that this $\varphi$ between polytopes is given by a $\Z$-linear map  
$\varphi:\Z[E(\Gamma)]\to\Z[E(\Gamma^\prime)]$.

In the algorithm for computing permutations for a ribbon graph, a cycle $C$ of the permutation $\sigma_i$ corresponds to a thick black circle.   Each circle is made up of some set $S$ of edges, and the length of the cycle is the sum of the weights of the edges in $S$.  The circles of the ribbon graph  correspond to the edges of the tropical graph.  Changing a traffic rule  cuts apart or joins together some circles, which corresponds to a vertex cutting or joining edges in the tropical graph, and so we have produced a monodromy graph from a ribbon graph.  Furthermore, if we want to know the edge weight of a given edge of the monodromy graph, it is clear it is the sum of all edge weights of the corresponding cycle on the ribbon graph, which is clearly a linear map.


\bibliographystyle{amsplain}
\bibliography{collected}
\end{document}